\documentclass[twoside,reqno]{amsart}
\newcommand\CC{\hbox{C\kern -.58em {\raise .54ex \hbox{$\scriptscriptstyle |$}}
  \kern-.55em {\raise .53ex \hbox{$\scriptscriptstyle |$}} }}
\newcommand\NN{\hbox{I\kern-.2em\hbox{N}}}
\newcommand\RR{\hbox{I\kern-.2em\hbox{R}}}
\newcommand\sRR{{\sl \hbox{I\kern-.2em\hbox{R}}}}
\newcommand\QQ{\hbox{I\kern-.53em\hbox{Q}}}
\newcommand\ZZ{{{\rm Z}\kern-.28em{\rm Z}}}

\theoremstyle{plain}
\newtheorem{lemma}{Lemma}[section]
\newtheorem{theorem}[lemma]{Theorem}
\newtheorem{corollary}[lemma]{Corollary}
\newtheorem{proposition}[lemma]{Proposition}
\newtheorem{definition}[lemma]{Definition}
\newtheorem*{proposition*}{Proposition}
\newtheorem*{definition*}{Definition}

\newtheorem{remark}[lemma]{Remark}

\usepackage{amsfonts}
\usepackage{amssymb}
\usepackage{enumerate}
\usepackage{amssymb}

\begin{document}
\title[The Cuntz semigroup, comparison and the ideal property]{The Cuntz semigroup, a Riesz type interpolation property, comparison
and the ideal property}
\author{Cornel Pasnicu}
\address{Department of Mathematics\\
The University of Texas at San Antonio\\
San Antonio, TX 78249, USA} \email{\bf Cornel.Pasnicu@utsa.edu}

\author{Francesc Perera}
\address{Departament de Matem\`atiques\\
Universitat Aut\`{o}noma de Barcelona\\
08193 Bellaterra (Barcelona), Spain}
\email{\bf perera@mat.uab.cat}


\maketitle

\begin{abstract}
We define a Riesz type interpolation property for the Cuntz
semigroup of a $C^*$-algebra and prove it is satisfied by the Cuntz semigroup of every $C^*$-algebra with the
ideal property. Related to this, we obtain two characterizations of
the ideal property in terms of the Cuntz semigroup of the
$C^*$-algebra. Some additional characterizations are proved in the
special case of the stable, purely infinite $C^*$-algebras, and two
of them are expressed in language of the Cuntz semigroup. We
introduce a notion of comparison of positive elements for every
unital $C^*$-algebra that has (normalized) quasitraces. We prove
that large classes of $C^*$-algebras (including large classes of
$AH$ algebras) with the ideal property have this comparison
property.
\end{abstract}

\section{Introduction}
\indent Elliott's classification program for separable, nuclear
$C^*$-algebras by discrete invariants including the $K$-theory is
one of the most important and successful research areas in Operator
Algebras (\cite{Ell:classprob}, \cite{EllToms:reg}; see also \cite{Ror:encyc}).
While it
is clear that not all the separable, nuclear $C^*$-algebras can be
classified, very large classes of $C^*$-algebras are known to be
classifiable (see, e.g., \cite{Ror:encyc}). Despite the success,
counterexamples to Elliott's conjecture in the simple case were
found by R{\o}rdam (\cite{Ror:simple}) and by Toms
(\cite{Toms:example} and \cite{Toms:classproblem}). In
\cite{Toms:classproblem} Toms used the Cuntz semigroup to
distinguish simple, nuclear $C^*$-algebras which cannot be
distinguished by the conventional Elliott invariant, where the Cuntz
semigroup of a $C^*$-algebra $A$ is a positively ordered, abelian
semigroup whose elements are equivalence classes of positive
elements in matrix algebras over $A$ (see Section 2 for details). It
then became clear that a further study of the Cuntz semigroup is
needed, and that, in fact, it is very important in Elliott's
classification program. On the other hand, as pointed out in
\cite{EllToms:reg}, the understanding of the regularity properties
of simple $C^*$-algebras, including the comparison of positive
elements, is essential in the new development of Elliott's program.
Extending the notion of comparison of positive elements to classes
of non-simple $C^*$-algebras---e.g., to the $C^*$-algebras with
{\it the ideal property}---and proving appropriate comparison
results for these classes is clearly a necessary and non-trivial
thing to do. In this paper we contribute to the study of the ideal
property, to the understanding of the structure of the Cuntz
semigroup of $C^*$-algebras (with this property), and prove a
certain type of comparison of positive elements for some classes of
$C^*$-algebras with the ideal property.\\
\indent A $C^*$-algebra is said to have {\it the ideal property} if
each of its ideals is generated (as an ideal) by its projections
({\it in this paper, by an ideal we mean a closed, two-sided ideal};
{\it the only exception here will be the Pedersen ideal, which is
the smallest dense algebraic two-sided ideal in the} $C^*${\it
-algebra}).
Note that every simple $C^*$-algebra with an approximate
unit of projections and every $C^*$-algebra of real rank zero
 (\cite{BroPed:realrank}) has the
 ideal property. The ideal
 property has been studied extensively by the first named author (alone or in collaboration),
 for example in \cite{Pas:ideal_prop_shape}, \cite{Pas:ideal_prop_AH},
 \cite{PasRor:tensIP}, \cite{PasRor:pirr0}
 and \cite{GonJiaLiPas:red1}, \cite{GonJiaLiPas:red2}.
 The ideal property is important in Elliott's classification program.
 It appeared first in Ken Stevens' Ph.D. thesis in which he classified by a  $K$-theoretical
 invariant a certain class of (non-simple) $AI$ algebras with the ideal property. In
 \cite{Pas:ideal_prop_shape} the first named author
 classified the $AH$ algebras with the ideal property and with slow dimension
 growth up to a shape equivalence and gave several characterizations of when an arbitrary $AH$
 algebra has the ideal property. Recall that a $C^*$-algebra $A$ is said to be an
 $AH$ algebra, if $A$ is the
 inductive limit $C^{*}$-algebra of: $$\begin{array}{l}
 A_{1}\stackrel{\phi_{1,2}}{\longrightarrow}A_{2}\stackrel{\phi_{2,3}}{\longrightarrow}A_{3}\stackrel{\phi_{3,4}}
 \longrightarrow\cdots\stackrel{\phi_{n - 1,n}}\longrightarrow A_{n}\stackrel{\phi_{n,n + 1}}\longrightarrow\cdots
 \end{array}$$ with
 $A_{n}=\bigoplus_{i=1}^{t_{n}}P_{n,i}M_{[n,i]}(C(X_{n,i}))P_{n,i}$,
 where the local spectra $X_{n,i}$ are finite, connected $CW$ complexes, $t_{n}$ and
 $[n,i]$ are strictly positive integers, and each $P_{n,i}$ is a projection of
 $M_{[n,i]}(C(X_{n,i}))$. In \cite{GonJiaLiPas:red1} and \cite{GonJiaLiPas:red2}, jointly with
 Gong, Jiang and Li, the first named author proved a
reduction theorem saying that
 every $AH$ algebra
 with the ideal property and with the dimensions of the local spectra uniformly
 bounded (i.e., with {\it no
 dimension growth}) can be written
 as an $AH$
 algebra with the ideal property with (special) local spectra of dimensions $\leq 3$. This result
 generalizes similar
 and strong reduction theorems for real rank zero $AH$ algebras---proved by Dadarlat
 (\cite{Dad:reduction}) and
 Gong (\cite{Gong:classII})---and also for
 simple $AH$ algebras---proved by Gong (\cite{Gong:simple_AH})---which have been major steps in
 the classification of the
 corresponding
 classes of $AH$ algebras. Also, in \cite{Pas:ideal_prop_shape} and \cite{Pas:ideal_prop_AH},
 the first named author proved several
nonstable $K$-theoretical results
 for a large class of $C^*$-algebras with the ideal property. Indeed, if $A$ is an $AH$ algebra with
 the ideal property
 and with slow dimension growth, it is proved in \cite{Pas:ideal_prop_shape} that $A$ has stable
 rank one (that means, in the unital case, that the set of the invertible
elements in $A$ is dense in $A$), that $K_0(A)$ is weakly
unperforated in the sense of Elliott and is also a Riesz group
 (\cite{Pas:ideal_prop_shape} and \cite{Pas:ideal_prop_AH}) and that the strict comparability
 of the projections in $A$ is determined by the
 tracial states of $A$, when $A$ is unital (\cite{Pas:ideal_prop_shape}). Also, jointly with
 R{\o}rdam, the first named author proved in \cite{PasRor:tensIP}
 that the ideal property is {\it not} preserved by taking minimal tensor products
 (even in the separable case). A
 characterization of the ideal property in the separable, purely
 infinite case is given by the first named author jointly with
 R{\o}rdam in \cite{PasRor:pirr0}, in terms of the Jacobson topology of the
 primitive spectrum of the $C^*$-algebra. The class of {\it purely infinite} $C^*$-algebras
 have been introduced by
Kirchberg and R{\o}rdam in \cite{KirRor:pi}, extending the
definition in the simple case given by Cuntz (\cite{Cuntz:KOn}). A
$C^*$-algebra $A$ is said to be purely infinite if $A$ has no
characters (or, equivalently, no non-zero abelian quotients) and if
for every $a, b \in A^+$ such that $a \in \overline{AbA}$
(the ideal of $A$ generated by $b$), it follows that there is a
sequence $\{x_n\}$ of elements in $A$ such that $a = \displaystyle
\lim_{n \rightarrow \infty} x_{n}^{*}bx_n$ (\cite{KirRor:pi}). The
study of purely infinite $C^*$-algebras was motivated by Kirchberg's
classification of the separable, nuclear $C^*$-algebras that
tensorially absorb the Cuntz algebra $\mathcal{O}_{\infty}$ up to
stable isomorphism by an ideal related $KK$-theory.\\
\indent The paper is divided into three sections. In Section 2 we remind the
reader some relevant definitions and notation---including the definition of
the Cuntz semigroup---
and define a Riesz type interpolation property for the Cuntz semigroup
$W(A)$ of a $C^*$-algebra $A$ (Definition \ref{dfn:2.1}). We prove that $W(A)$ has this property, whenever $A$ is a $C^*$-algebra the ideal property (Theorem \ref{thm:weakRiesz}).
Related to this, we obtain two characterizations of the ideal property in
terms of the Cuntz semigroup of the $C^*$-algebra (Theorem \ref{thm:2.5}). Some
additional characterizations are proved in the special case of the stable,
purely infinite $C^*$-algebras, and two of them are expressed in language
of the Cuntz semigroup of the algebra (Theorem \ref{thm:2.6}). In Section 3, after
reminding the reader some definitions, notation and results, we introduce
a notion of comparison of positive elements for every unital $C^*$-algebra
that has (normalized) quasitraces (Definition \ref{dfn:3.1}). We prove that large
classes of $C^*$-algebras (including large classes of $AH$ algebras)
with the ideal property have this comparison property (Theorems \ref{thm:3.5}, \ref{thm:3.6} and \ref{thm:3.7}).\\
\indent The symbol $\otimes$ will mean the minimal tensor product
of $C^*$-algebras.

\section{A Riesz type interpolation property for the Cuntz semigroup
and the ideal property}

%\indent We will define two Riesz type interpolation properties for
%the Cuntz semigroup of a $C^*$-algebra and will prove that they are
%satisfied in ``many" cases. Related to this, we will obtain two
%characterizations of the ideal property in terms of the Cuntz
%semigroup of the $C^*$-algebra. Some additional characterizations in
%language of the Cuntz semigroup will be proved in the special case
%of stable
%and purely infinite $C^*$-algebras.\\
\indent We start by recalling mainly some definitions and
notation (see also \cite{AraPerTom:survey}). If $A$ is a $C^*$-algebra and $a, b\in A^{+}$, then we
write $a \precsim b$ if there is a sequence $\{x_{n}\}$ of elements
of $A$ such that $a = \displaystyle \lim _{n \rightarrow \infty}
x_{n}bx_{n}^{*}$. This relation can be extended to the (local)
$C^*$-algebra $M_{\infty}(A)$ defined as the algebraic inductive
limit of $M_{n}(A)$ via the inclusion mappings $M_{n}(A)
\hookrightarrow M_{n + 1}(A)$ given by $x \mapsto x \oplus 0$. Let
$M_{\infty}(A)^{+}$ denote the set of positive elements of
$M_{\infty}(A)$. If $a, b \in M_{\infty}(A)^{+}$, we write $a
\precsim b$ provided that $a \precsim b$ in $M_{n}(A)$ for some $n$
such that $a, b \in M_{n}(A)$. (If we view $a$ and $b$ in two
different sized matrices over $A$, the above is equivalent to having
$a = \displaystyle \lim _{n \rightarrow \infty} x_{n}bx_{n}^{*}$
where the $x_{n}$ are suitable rectangular matrices.) If both $a
\precsim b$ and $b \precsim a$, we will write $a \sim b$ and will
call $a$ and $b$ Cuntz equivalent (see \cite{Cuntz:dimension}). We shall denote $W(A) =
M_{\infty}(A)^{+} / \sim$, and $\langle a \rangle \in W(A)$ will
denote the Cuntz equivalence class of an element $a$ of
$M_{\infty}(A)^{+}$ (so that $W(A) = \{ \langle a \rangle: a \in
M_{\infty}(A)^{+} \}$). Then $W(A)$ is a positively ordered abelian
semigroup when equipped with the relations:
\[
\langle a \rangle + \langle b \rangle = \langle a \oplus b \rangle,
\,\, \langle a \rangle \leq \langle b \rangle \Leftrightarrow a
\precsim b, \,\, a, b \in M_{\infty}(A)^{+}\,.
\]
\indent We shall refer to $W(A)$ as the {\it Cuntz semigroup of}
$A$. One shortcoming of this construction is that this semigroup fails to be
continuous with respect to sequential inductive limits. This was remedied in
\cite{CowEllIva:Cuntz} by constructing an ordered semigroup, termed $Cu(A)$, in
terms of countably generated Hilbert modules. This new object turned out to be
intimately related to $W(A)$, in that $Cu(A)$ is order isomorphic to
$W(A\otimes\mathcal K)$, where $\mathcal{K}$
is the $C^*$-algebra of the compact operators on $\ell^{2}(\mathbb{N})$
(\cite{CowEllIva:Cuntz}). The semigroup $Cu(A)$ is, as opposed to
$W(A)$, closed under (order-theoretic) suprema of increasing sequences and, in
significant cases, can be regarded as a completion of $W(A)$ (see
\cite{CowEllIva:Cuntz}, \cite{AntBosaPer:compl}). Moreover, every element in $Cu(A)$ is a supremum of a \emph{rapidly increasing} sequence.
More precisely, we write $x\ll y$ to mean that whenever $y\leq \sup y_n$ for an increasing sequence $\{y_n\}$, there is $m$ such that $x\leq y_m$. A sequence $\{x_n\}$ is rapidly increasing
if $x_n\ll x_{n+1}$ for all $n$. It is shown in \cite{CowEllIva:Cuntz} that, for any positive $a$, the sequence $\{\langle (a-1/n)_+\rangle\}$ is rapidly increasing (and with supremum $\langle a\rangle$).\\
% For the definition of $Cu(A)$ and related issues see
% \cite{CowEllIva:Cuntz} (and also \cite{AraPerTom:survey}). Note that
% $Cu(A) = W(A \otimes \mathcal{K})$ (see \cite{CowEllIva:Cuntz},
% \cite{AraPerTom:survey}), where $\mathcal{K}$
% is the $C^*$-algebra of the compact operators on $\ell^{2}(\mathbb{N})$.\\
\indent If $A$ is a (local) $C^*$-algebra, then denote the set of projections
in $A$ by $\mathcal{P}(A) := \{p \in
A: p = p^{2} = p^{*}\}$.\\
% and $\mathcal{P}(M_{\infty}(A)) := \{p \in
% M_{\infty}(A): p = p^{2} = p^{*}\}$.\\
\indent For each $a \in A^{+}$ and for each $\varepsilon > 0$, we
shall write $(a - \varepsilon)_{+}$ for the positive element of $A$
given by $h_{\varepsilon}(a)$, where $h_{\varepsilon}(t)$ = max$\{t
-\varepsilon, 0\}$.
\begin{definition}
\label{dfn:2.1}
Let $A$ be a $C^*$-algebra. We say
that the Cuntz semigroup $W(A)$ has the \emph{weak Riesz
interpolation by projections property} if for all $a_{i}, b_{i} \in
M_{\infty}(A)^{+}$ such that $\langle a_{i}\rangle \leq \langle
b_{j}\rangle$ (in $W(A)$), $1 \leq i,j \leq 2$ and for every $\varepsilon > 0$,
there exist a projection $p \in \mathcal P (M_{\infty}(A))$ and
$m \in \NN$ such that $\langle (a_{i} - \varepsilon)_{+}
\rangle \leq \langle p\rangle \leq m\langle b_{j}\rangle$, $1 \leq
i,j \leq 2$ (in $W(A)$).
\end{definition}
This property just defined is of course related to the property of Riesz interpolation by general positive elements. This can be achieved in the real rank zero setting
(\cite{per:ijm}), as well as for simple stable algebras that absorb the Jiang-Su algebra (\cite{Tiku:ijm}).

We want to prove the following result:
\begin{theorem}\label{thm:weakRiesz}
Let $A$ be a $C^*$-algebra with the ideal property. Then $W(A)$
has the weak Riesz interpolation by projections property.
\end{theorem}

The proof of the above theorem will use the following, of which part (ii) is known and essentially contained in \cite{KirRor:pi}; we include it here
just for convenience.
\begin{lemma}
\label{lem:2.4}
Let $A$ be a $C^*$-algebra, let $I$ be an ideal of $A$ that is
generated (as an ideal) by $\mathcal{P}(I)$ and let $a \in A^{+}$.
\begin{enumerate}[{\rm (i)}]
\item If $a \in I$, then for every $\varepsilon >
0,$, there exists $p \in \mathcal{P}(M_{\infty}(A))$ such that
$(a - \varepsilon)_{+} \precsim p$, where $p$ is a finite direct sum
of projections of $I$.
\item For every $q \in \mathcal{P}(\overline{AaA})$, there exists $n
\in \mathbb{N}$ such that $q \precsim a \otimes 1_{n}$.
\end{enumerate}
\end{lemma}
\begin{proof}
(i). Since $I$ is generated by its projections and the element $(a-\varepsilon)_+$ belongs to the Pedersen ideal of $I$, we can write
$(a-\varepsilon)_+=\sum_{i=1}^k x_ip_iy_i$, where $x_i$, $y_i\in A$ and $p_i$ are projections from
$I$.  Next,
\[
(a - \varepsilon)_{+} = \sum_{i=1}^k x_ip_iy_i = \sum_{i=1}^k y_i^*p_ix_i^* = \frac{1} {2} \sum_{i=1}^k (x_ip_iy_i  +  y_i^*p_ix_i^*) \leq \frac{1} {2} \sum_{i=1}^k (x_ip_ix_i^*  +  y_i^*p_iy_i)\,,
\]
from which the conclusion follows using  \cite[Lemma 2.5(ii) and
Lemma 2.8(ii)] {KirRor:pi}.

(ii). Fix some $0 <\varepsilon < 1$. Since $q \in
\mathcal{P}(\overline{AaA}),$ using \cite[Proposition
2.7(v)]{KirRor:pi}, we get for this $\varepsilon > 0$ that there exists some $n
\in \mathbb{N}$ such
that:
\[
q \sim (q - \varepsilon)_{+} \precsim a \otimes
1_{n}\,,
\]
as wanted.
\end{proof}

\indent{\it Proof of Theorem \ref{thm:weakRiesz}}. Let $a_{i}, b_{i} \in
M_{\infty}(A)^{+}$ such that $\langle a_{i}\rangle \leq \langle
b_{j}\rangle$, $1 \leq i,j \leq 2$. We may suppose that $a_{i},
b_{i} \in A^{+}, 1 \leq i \leq 2$. Let $\varepsilon > 0$. Note that
\[
\langle a_i\rangle\leq \langle a_1+a_2\rangle
\]
for $i=1,2$.
%
%Then, by
%Proposition \ref{prop:2.3} there exists $c\in A^+$ (which may be taken as
%$(a_{1} + a_{2} -\delta_{1})_{+}$, for some $0 < \delta_{1} \leq \varepsilon)$
%and there exists $k \in \mathbb{N}$ such that:
%\begin{equation}
%\langle (a_{i} - \frac{\varepsilon}{2})_{+}\rangle \leq \langle c
%\rangle \leq k\langle b_{j}\rangle, 1 \leq i, j \leq 2\label{2.4}
%\end{equation}
Then, by \cite{Ror:uhfII}, for our
$\varepsilon> 0$, there exists $\delta > 0$ such that:
\begin{equation}
\langle (a_{i} - \varepsilon)_{+}\rangle \leq
\langle (c - \delta)_{+}\rangle, 1 \leq i \leq 2\,,\label{2.5}
\end{equation}
where $c=a_1+a_2$. Since $\langle a_i\rangle \leq \langle b_{j}\rangle, 1
\leq i, j \leq 2$, we have that $c \in \overline{Ab_{j}A}, 1
\leq j \leq 2$, i.e. $c \in I := \overline{Ab_{1}A} \cap
\overline{Ab_{2}A}$. Note that since $A$ has the ideal property and
$I$ is an ideal of $A$, it follows that $I$ is generated (as an
ideal) by $\mathcal P(I)$. Then, by Lemma \ref{lem:2.4}, it follows that
for our $\delta > 0$ there exists $p \in \mathcal{P}(M_{\infty}(A))$
such that $p$ is a finite direct sum of projections of $I$ and there
exists $m \in \mathbb{N}$ such that:
\begin{equation}
\langle (c - \delta)_{+}\rangle \leq \langle p \rangle \leq m\langle
b_{j} \rangle, 1 \leq j \leq 2\label{2.6}
\end{equation}
\indent Finally, \eqref{2.5} and \eqref{2.6} imply that:
\[
\langle (a_{i} - \varepsilon)_{+}\rangle \leq \langle p \rangle \leq
m\langle b_{j} \rangle, 1 \leq i, j \leq 2\,,
\]
which ends the proof.\qed\\

\indent We now prove a theorem inspired by the proof of Theorem
\ref{thm:weakRiesz} and that gives two characterizations of the ideal property
in terms of the Cuntz semigroup of the
$C^*$-algebra. In particular, it implies Theorem \ref{thm:weakRiesz}.
\begin{theorem}
\label{thm:2.5} Let $A$ be a $C^*$-algebra. Then, the following are
equivalent:
\begin{enumerate}[{\rm (i)}]
\item $A$ has the ideal property;
\item For all $a_{i}, b_{i} \in A^{+}$ such that
$\langle a_{i}\rangle \leq \langle b_{j}\rangle$, $1 \leq i,j \leq
2$ and for every $\varepsilon > 0$, there exist a projection $p \in
\mathcal P (M_{\infty}(A))$ and $m \in \mathbb{N}$ such that
$\langle
(a_{i} - \varepsilon)_{+} \rangle \leq \langle p\rangle \leq
m\langle b_{j}\rangle$,
$1 \leq i,j \leq 2$ and $p$ is a finite direct sum of projections of
$A$;
\item For every $a \in A^{+}$ and every $\varepsilon > 0$, there are a
projection $p \in \mathcal P (M_{\infty}(A))$ and $m \in
\mathbb{N}$ such that $\langle (a - \varepsilon)_{+} \rangle \leq
\langle p\rangle \leq m\langle a\rangle$ and $p$ is
a finite direct sum of
projections of $A$.
\end{enumerate}
\end{theorem}
\begin{proof}
\indent (i) $\Rightarrow$ (ii). It follows from the
proof of Theorem \ref{thm:weakRiesz}.

\indent (ii) $\Rightarrow$ (iii). Let $a$ be an arbitrary element of
$A^{+}$. Then choose $a_{i} = b_{i} = a$, $1 \leq i \leq 2$ and use
(ii).

\indent (iii) $\Rightarrow$ (i). Let $I$ be an ideal of $A$, $a \in I
\cap A^{+}$ and $\varepsilon > 0$. Then, by (iii), there exist $n \in
\mathbb{N}$, $p_{1}, p_{2}, \ldots, p_{n} \in \mathcal {P} (A)$ and there
exists $m \in \mathbb{N}$ such that:
\begin{equation}
\langle (a - \varepsilon)_{+}\rangle \leq \langle \oplus^{n}_{k = 1}
p_{k} \rangle \leq m\langle a \rangle\label{2.7}
\end{equation}
\indent Increasing $m$ and then $n$, if necessary, (defining the new
$p_{k}'s$ to be 0) we may suppose that $m = n$ in \eqref{2.7}. Working in
$M_{n}(A)$, it is easy to see that the first inequality in \eqref{2.7}
implies that:
\begin{equation}
(a - \varepsilon)_{+} \in \overline{Ap_{1}A + Ap_{2}A + \cdots +
Ap_{n}A}\label{2.8}
\end{equation}
\noindent while the second inequality in \eqref{2.7} implies immediately
that:
\begin{equation}
p_{k} \in \overline{AaA}, 1 \leq k \leq n\label{2.9}
\end{equation}
\indent Observe that \eqref{2.8} and \eqref{2.9} imply that $(a -
\varepsilon)_{+}$ belongs to the ideal of $A$ generated by
$\mathcal{P}(\overline{AaA})\subseteq \mathcal{P}(I)$. Therefore, $a
= \displaystyle \lim_{\varepsilon \rightarrow 0} (a -
\varepsilon)_{+}$ belongs to the ideal of $A$ generated by
$\mathcal{P}(I)$ and hence, because each element of $I$ is a linear
combination of four positive elements of $I$, $I$ is generated as an
ideal of $A$ by $\mathcal{P}(I)$. Since $I$ is an arbitrary ideal of
$A$, this proves that $A$ has the ideal
property.
\end{proof}
The result below gives some additional characterizations of the
ideal property in the case of stable, purely infinite
$C^*$-algebras. Two of these characterizations are in terms of the
Cuntz semigroup of the $C^*$-algebra. As already mentioned in the
introduction, a $C^*$-algebra $A$ is termed \emph{purely infinite}
if $A$ does not have non-zero abelian quotients and $a\precsim b$
whenever $a\in \overline{AbA}$ (see \cite{KirRor:pi}). Recall also
that a positive, non-zero element $a$ of a $C^*$-algebra $A$ is said
to be \emph{properly infinite} if $a \oplus a \precsim a \oplus 0$
in $M_{2}(A)$. It was shown in \cite[Theorem 4.16]{KirRor:pi} that a
$C^*$-algebra $A$ is purely infinite if and only if all non-zero
positive elements are properly infinite.

\begin{theorem}
\label{thm:2.6} Let $A$ be a purely infinite,
stable $C^*$-algebra. Then, the following are equivalent:
\begin{enumerate}[{\rm (i)}]
\item $A$ has the ideal property;
\item For every $a \in A^{+}$, there exists a sequence $\{p_{n}\}$ of
projections in $A$ such that
$\langle a \rangle = \displaystyle \sup_{n \in \mathbb{N}} \langle p_{n}
\rangle$ (in $W(A)$);
\item For every $a \in A^{+}$, there exists a sequence
$\{q_{n}\}$ of projections in $A$ such that $\{\langle
q_{n} \rangle\}$ is increasing in $W(A)$ and
$\langle a \rangle = \displaystyle \sup_{n \in \mathbb{N}} \langle q_{n}
\rangle$ (in $W(A)$);
\item For all $a \in A^{+},$ we have that
$\overline{AaA} = \overline{\cup_{n \geq 1}I_{n}},$ where $\{I_{n}\}$
is an increasing sequence of ideals of
$A$ and each $I_{n}$ is generated (as an ideal) by a
single projection.
\end{enumerate}
\end{theorem}

The proof of the above theorem will use the following:

\begin{proposition}
\label{prop:2.7} Let $A$ be a $C^*$-algebra with the ideal property and let $a
\in A^{+}$.
\begin{enumerate}[{\rm (i)}]
\item If a is either properly infinite or zero,
then there is a sequence $\{p_{n}\}$ of projections in
$M_{\infty}(A)$ such that $\langle a \rangle = \displaystyle
\sup_{n \in \mathbb{N}}
\langle p_{n} \rangle$ (in $W(A)$).
\item If $(a -\varepsilon)_{+}$ is either properly
infinite or zero for every $\varepsilon > 0$, then there is a
sequence $\{q_{n}\}$ of projections in $M_{\infty}(A)$
such that $\{\langle q_{n}\rangle\}$ is an increasing sequence
in $W(A)$ and $\langle a \rangle = \displaystyle \sup_{n \in
\mathbb{N}} \langle q_{n}
\rangle$  (in $W(A)$).
\end{enumerate}
\end{proposition}
\begin{proof} (i). Let $\{\varepsilon_{n}\}$ be a strictly
decreasing sequence of strictly positive numbers such that
$\displaystyle \lim_{n \rightarrow \infty} \varepsilon_{n} = 0$.
Since $A$ has the ideal property, Theorem \ref{thm:2.5} implies that there
exists a sequence $\{p_{n}\}$ of projections of $M_{\infty}(A)$ such
that:
\begin{equation}
\langle(a - \varepsilon_{n})_{+}\rangle \leq \langle p_{n}\rangle
 \leq \langle a
\rangle, \text{ for all } n \in \mathbb{N}\label{2.10}
\end{equation}
\noindent (we also used the fact that $m\langle a \rangle = \langle
a \rangle$, for every $m \in \mathbb{N}$, since {\it a} is either
properly infinite or zero). Let us prove now that $\displaystyle
\sup_{n \in \mathbb{N}} \langle p_{n} \rangle = \langle a \rangle$.
For this, observe first that by \eqref{2.10} we have $\langle p_{n}
\rangle \leq \langle a \rangle$ for every $n \in \mathbb{N}$. Now,
let $x \in W(A)$ be such that $\langle p_{n} \rangle \leq x$, for
every $n \in \mathbb{N}$. Then, \eqref{2.10} implies that $\langle (a -
\varepsilon_{n})_{+}\rangle \leq x$, for every $n \in \mathbb{N}$.
Therefore, since also $\{\varepsilon_{n}\}$ is a strictly decreasing
sequence of strictly positive numbers and $\displaystyle \lim_{n
\rightarrow \infty} \varepsilon_{n} = 0$, we have:
\[
\langle a \rangle = \displaystyle \sup_{n \in \mathbb{N}} \langle (a
- \varepsilon_{n})_{+} \rangle \leq x\,.
\]
%\noindent (see, e.g., \cite[Lemma 4.36]{AraPerTom:survey}).
This
ends the proof of the
equality $\displaystyle \sup_{n \in \mathbb{N}} \langle p_{n} \rangle = \langle
a \rangle$.

\indent (ii). Let $\{\varepsilon_{n}\}$ be a strictly decreasing
sequence of strictly positive numbers such that $\displaystyle \lim_{n
\rightarrow \infty} \varepsilon_{n} = 0$. Then, since $\varepsilon_{n} - \varepsilon_{n
+ 1} > 0$ and
\[
(a - \varepsilon_{n})_{+} = ((a - \varepsilon_{n +
1})_{+} - (\varepsilon_{n} - \varepsilon_{n + 1}))_{+}
\]
for every $n \in \mathbb{N}$, Theorem \ref{thm:2.5} implies that
there exists a sequence $\{q_{n}\}$ of projections of
$M_{\infty}(A)$ such that:
\begin{equation}
\langle(a - \varepsilon_{n})_{+}\rangle \leq \langle q_{n}\rangle
\leq \langle(a - \varepsilon_{n + 1})_{+}\rangle \leq \langle a
\rangle, \text{ for all } n \in \mathbb{N}\label{2.11}
\end{equation}
\noindent (we also used the fact that $m\langle (a - \varepsilon_{n
+ 1}) \rangle = \langle (a - \varepsilon_{n + 1}) \rangle$, for
every $m, n \in \mathbb{N}$, since, by hypothesis, $(a -
\varepsilon_{n + 1})_{+}$ is either properly infinite or zero for
all $n \in \mathbb{N}$). Clearly, \eqref{2.11} implies that $\{\langle
q_{n} \rangle\}$ is an increasing sequence in $W(A)$. Finally, let
us observe that $\displaystyle \sup_{n \in \mathbb{N}} \langle q_{n}
\rangle = \langle a \rangle$ in $W(A)$, since its proof is similar
with that of the fact that $\displaystyle \sup_{n \in \mathbb{N}}
\langle p_{n} \rangle = \langle a \rangle$ in $W(A)$ in the above
proof of part (i) of this proposition.
\end{proof}

\begin{corollary}
\label{cor:2.8}
Let $A$ be a purely
infinite $C^*$-algebra with the ideal property and let $a \in
A^{+}$.  Then, there exists a sequence $\{p_{n}\}$ of
projections in $M_{\infty}(A)$ such that $\{\langle p_{n} \rangle\}$  is
an increasing sequence in $W(A)$ and $\langle a \rangle =
\displaystyle \sup_{n \in \mathbb{N}} \langle p_{n} \rangle$ (in $W(A)$).
\end{corollary}
\begin{proof}
Since $A$ is purely infinite, \cite[Theorem
4.16]{KirRor:pi} implies that every positive element of $A$ is
either properly infinite or zero. The proof ends now applying condition (ii) in
Proposition \ref{prop:2.7}.
\end{proof}

\indent{\it Proof of Theorem \ref{thm:2.6}} (i) $\Rightarrow$ (iii). It follows
immediately from Corollary \ref{cor:2.8}.\\
\indent (iii) $\Rightarrow$ (ii). The proof of this implication is trivial.\\
\indent (ii) $\Rightarrow$ (iv). Assume (ii). Let $a \in A^{+}$. Then,
by (ii), there exists a sequence $\{p_{n}\}$ of projections of $A$
such that:
\begin{equation}
\langle a \rangle = \displaystyle \sup_{n \in \mathbb{N}} \langle
p_{n} \rangle\label{2.12}
\end{equation}
\indent Since $A$ is stable, replacing each $p_{n}$ by a projection
in its Murray-von Neumann equivalence class we may suppose that
$p_{m}p_{n} = 0$ for every $m \neq n, m, n \in \mathbb{N}$. Then, for every
$n \in \mathbb{N}$ we have that $p_{1} + p_{2} + \cdots + p_{n} \in
\mathcal{P}(A)$ and $p_{n} \leq p_{1} + p_{2} + \cdots + p_{n}$ and
hence, using \cite[Lemma 2.8(iii) and Lemma 2.9]{KirRor:pi} we have:
\begin{equation}
\langle p_{n} \rangle \leq \langle p_{1} + p_{2} + \cdots + p_{n}
\rangle = \sum^{n}_{i = 1} \langle p_{k} \rangle \leq n\langle a
\rangle = \langle a \rangle, \text{ for all } n \in \mathbb{N}\label{2.13}
\end{equation}
\noindent since $\langle p_{k} \rangle \leq \langle a \rangle$ for
all $k \in \mathbb{N}$ and $a$ is properly infinite or zero (by
\cite[Theorem 4.16]{KirRor:pi}). Define:
\[
q_{n} := p_{1} + p_{2} + \cdots + p_{n} (\in \mathcal{P}(A)),
\text{ for all } n \in \mathbb{N}
\]
\indent Since $p_{1} + p_{2} + \cdots + p_{n} \leq p_{1} + p_{2} +
\cdots + p_{n + 1}$, for all $n \in \mathbb{N}$, it follows that $\{\langle
q_{n} \rangle\}$ is an increasing sequence in $W(A)$. Also, \eqref{2.12}
and \eqref{2.13} imply that\
\begin{equation}
\langle a \rangle = \displaystyle \sup_{n \in \mathbb{N}} \langle
q_{n}\rangle\label{2.14}
\end{equation}
\indent We want to prove that:
\begin{equation}
\overline{AaA} = \overline{{\cup}_{n \geq 1} {Aq_{n}A}} (=
\overline{{\cup}_{n \geq 1}{\overline{Aq_{n}A}}})\label{2.15}
\end{equation}
\noindent (note that $\langle q_{n} \rangle  \leq \langle q_{n + 1}
\rangle $ implies that $\overline{Aq_{n}A} \subseteq \overline{Aq_{n + 1}A}$,
for all $n \in \mathbb{N}$). Indeed, observe that \eqref{2.14} implies that
$\langle q_{n} \rangle \leq \langle a \rangle$, for all $n \in \mathbb{N}$,
from which we get $q_{n} \in \overline{AaA}$ for all $n \in \mathbb{N}$,
and hence:
\begin{equation}
\overline{{\cup}_{n \geq 1} {Aq_{n}A}} \subseteq
\overline{AaA}\label{2.16}
\end{equation}
\indent Let $\varepsilon > 0$. Since $\langle(a -
\varepsilon)_{+}\rangle \ll \langle a \rangle$ in $Cu(A) = W(A)$
(take into account the fact that $A$ is stable and see
\cite{CowEllIva:Cuntz}), from (2.14) we deduce that:
\[
\langle(a - \varepsilon)_{+} \rangle \leq  \langle q_{m}
\rangle
\]
for some $m \in \mathbb{N}$, which implies that $(a
-\varepsilon)_{+} \in \overline{Aq_{m}A} \subseteq
\overline{{\cup}_{n \geq 1} {Aq_{n}A}}$. Since $\varepsilon > 0$ is
arbitrary, we have that $a = \displaystyle \lim_{\varepsilon
\rightarrow 0}(a -\varepsilon)_{+} \in \overline{{\cup}_{n \geq 1}
{Aq_{n}A}}$, and therefore:
\begin{equation}
\overline{AaA} \subseteq \overline{{\cup}_{n \geq 1} {Aq_{n}A}}(=
\overline{{\cup}_{n \geq 1}{\overline{Aq_{n}A}}})
\label{2.17}
\end{equation}
\indent Observe that \eqref{2.16} and \eqref{2.17} prove the equality \eqref{2.15}.
Hence, if we define $I_{n} := \overline{Aq_{n}A}$ for every $n \in
\mathbb{N}$, we have that $\overline{AaA} = \overline{\cup_{n \geq 1}I_{n}}$,
where $\{I_{n}\}$ is an increasing sequence of ideals of $A$ and
each $I_{n}$ is generated (as an ideal) by a single projection
(namely $q_{n}$).\\
\indent (iv) $\Rightarrow$ (i). Assume (iv). Let $I$ be an ideal of
$A$ and let $a \in I^{+}$. Then, by (iv) we have that $\overline{AaA}
= \overline{\cup_{n \geq 1}I_{n}},$ where $\{I_{n}\}$ is an increasing
sequence of ideals of $A$ and each $I_{n}$ is generated (as an
ideal) by a single projection. This implies that $a$ belongs to the
ideal of $A$ generated by $\mathcal{P}(\overline{AaA}) \subseteq
\mathcal{P}(I)$, and hence $a$ belongs to the ideal of $A$ generated
by $\mathcal{P}(I)$. Since $a \in I^{+}$ was arbitrary and each
element of $I$ is a linear combination of four positive elements of
$I$, it follows that $I$ is  generated (as an ideal) by
$\mathcal{P}(I)$. But $I$ was an arbitrary ideal of $A$, so we
deduce that $A$ has the ideal property.\qed

\begin{remark}
\begin{enumerate}[{\rm (i)}]
\item {\rm Note that if $A$ is a purely infinite
$C^*$-algebra, then $W(A)$ has {\it the Riesz interpolation
property}. The same conclusion holds for the semigroup $V(A)$ consisting of the Murray-von Neumann equivalence classes $[p]$ of projections in $M_{\infty}(A)$. Indeed, let $a_{i}, b_{i} \in M_{\infty}(A)^{+}$ be such that $\langle a_{i} \rangle \leq \langle b_{j} \rangle, 1 \leq i, j \leq 2$ (in $W(A)$). We may assume that $a_{i}, b_{i} \in A^{+}, 1 \leq i \leq 2$.
Then, for all $i$, $j$
\[
\langle a_i\rangle\leq \langle a_1+a_2\rangle\leq\langle a_1\rangle+\langle a_2\rangle\leq 2\langle b_j\rangle\leq\langle b_j\rangle\,.
\]
(We have used here that every non-zero positive element is properly infinite.)}

\item {\rm For a $C^*$-algebra $A$, denote by
\[
W_{pi}(A)=\{\langle a\rangle\in W(A)\mid a=0\text{ or else properly infinite in }M_{\infty}(A)\}\,.
\]
Then the same argument as in (i) shows that $W_{pi}(A)$ is a subsemigroup of $W(A)$ with Riesz interpolation. With this language, \cite[Theorem 4.16]{KirRor:pi} can be rephrased by saying that $A$ is purely infinite if and only if $W(A)=W_{pi}(A)$.}
\end{enumerate}
\end{remark}

\section{Comparison of positive elements and the ideal property}
\indent We will introduce a notion of comparison of positive
elements for unital $C^*$-algebras that have (normalized)
quasitraces. We will prove that large classes of $C^*$-algebras with
the ideal property have this comparison property.\\
\indent We begin recalling some definitions, notation and results.
The notion of {\it dimension function} was introduced by Cuntz in
\cite{Cuntz:dimension}. A dimension function on a $C^*$-algebra $A$
is an additive order preserving function $d : W(A) \rightarrow [0,
\infty]$. We can also regard $d$ as a function $M_{\infty}(A)^{+}
\rightarrow [0, \infty]$ that respects the rules $d(a \oplus b) =
d(a) + d(b)$ and $a \precsim b \Rightarrow d(a) \leq d(b)$ for all
$a, b \in M_{\infty}(A)^{+}$. The set of all dimension functions on
a $C^*$-algebra $A$ will be denoted by $DF(A)$. A dimension function
$d$ on $A$ is said to be {\it lower semicontinuous} if $d(a) =
\displaystyle \sup_{\varepsilon > 0} d((a - \varepsilon)_{+})$
for all $a \in M_{\infty}(A)^{+}$.\\
\indent Let $A$ be a unital $C^*$-algebra. A (normalized) quasitrace on $A$ ia a function
$\tau: A \rightarrow \mathbb{C}$ satisfying:\\[2ex]
\indent(i) \,\,\,$\tau (1) = 1,$\\
\indent(ii) \,\,$0 \leq \tau(xx^{*}) = \tau(x^{*}x)$, for all $x \in A$,\\
\indent(iii) \,$\tau(a + ib) = \tau(a) + i\tau(b)$, for all $a, b \in A_{sa}$,\\
\indent(iv) \,$\tau$ is linear on abelian sub-$C^*$-algebras of $A$,\\
\indent(v) \,\,\,$\tau$ extends to a function from $M_{n}(A)$ to $\mathbb{C}$ satisfying (i)-(iv).\\[2ex]
\indent The set of all (normalized) quasitraces on $A$ will be
denoted $QT(A)$. This notion was introduced in
\cite{BlaHan:quasitrace}. Given $\tau \in QT(A)$ one may define a
map $d_{\tau} : M_{\infty}(A)^{+} \rightarrow [0, \infty]$ by:
\[
d_{\tau}(a) = \lim_{n \rightarrow \infty} \tau(a^{1/n})
\]
\indent Note that in fact $d_{\tau}$ {\it takes only real values}:
$d_{\tau}(M_{\infty}(A)^{+}) \subseteq [0, \infty)$. Blackadar and
Handelman showed in \cite{BlaHan:quasitrace} that $d_{\tau}$ is a
lower semicontinuous dimension function on $A$. Note that for
all $p \in \mathcal{P}(M_{\infty}(A))$ we have that $d_{\tau}(p) = \tau(p)$.
\begin{definition}
\label{dfn:3.1}
A unital $C^*$-algebra $A$ such that $QT(A) \neq \emptyset$ is
said to have \emph{weak strict comparison} if it has the property that $a \precsim b$ whenever
$a, b \in M_{\infty}(A)^{+}$ satisfy the inequality $d(a) < d(b)$ for every
$d \in E \cup \{f \in DF(A) \smallsetminus E:f(b) = 1\}$, where $E := \{d_{\tau}:\tau \in QT(A)\}$.
\end{definition}
\begin{definition}
A unital $C^*$-algebra $A$ such that $QT(A) \neq \emptyset$ is
said to have \emph{strict comparison of projections} if $p \precsim q$ whenever
$p, q \in \mathcal{P}(M_{\infty}(A))$ satisfy the inequality $\tau(p) < \tau(q)$ for every
$\tau \in QT(A)$.
\end{definition}
\begin{theorem}
\label{thm:3.3}
Let $A$ be a unital $C^*$-algebra with the ideal property.
Assume moreover that $A$ has strict comparison of projections and that it has finitely many extremal
quasitraces. Let $a, b \in M_{\infty}(A)^{+}$ such that:
\[
d_{\tau}(a) < d_{\tau}(b), \text{ for all } \tau \in QT(A)\,.
\]
Then, for every $\varepsilon > 0$, there is $m \in \mathbb{N}$ such that
$(a - \varepsilon)_{+} \precsim b \otimes 1_{m}$.
\end{theorem}
\begin{proof}
We may assume, without loss of generality, that $a, b \in A^{+}$.
Let $\tau_{1}, \tau_{2}, \ldots, \tau_{l}$ be the extremal
quasitraces of $A$, for some $l \in \mathbb{N}$. (Hence, each
quasitrace of $A$ is a convex combination of $\tau_i$, $1 \leq i
\leq l$.) Since $d_{\tau_{i}}$ is a lower semicontinuous dimension
function on $A$, we have that $d_{{\tau_{i}}}(b) = \displaystyle
\sup_{\delta > 0} d_{{\tau_{i}}}((b - \delta)_{+}), 1 \leq i \leq
l$. Hence, $d_{{\tau_{i}}}(a) < d_{{\tau_{i}}}(b)$ implies that
there exists $\varepsilon_{i} > 0$ such that:
\[
d_{{\tau_{i}}}(a) < d_{{\tau_{i}}}((b - \varepsilon_{i})_{+}), 1 \leq i \leq l\,.
\]
\indent Choose $\varepsilon_{0} > 0$ such that $\varepsilon_{0} \leq
\varepsilon_{i}, 1 \leq i \leq l$. Since $0 < \delta_{1} \leq
\delta_{2}$ implies $(b - \delta_{2})_{+} \leq (b - \delta_{1})_{+}$
and hence $(b - \delta_{2})_{+} \precsim (b - \delta_{1})_{+}$ from
which one obtains $d_{\tau}((b - \delta_{2})_{+}) \leq d_{\tau}((b -
\delta_{1})_{+})$, for every $\tau \in QT(A)$, the above
inequalities imply that:
\begin{equation}
d_{{\tau_{i}}}(a) < d_{{\tau_{i}}}((b - \varepsilon_{0})_{+}), 1 \leq i \leq l\,.\label{3.1}
\end{equation}
\indent On the other hand, since $A$ has the ideal property,
Lemma \ref{lem:2.4} implies that for $\varepsilon_{0} > 0$ there exist
projections $p$ and $q$ in $\mathcal{P}(M_{\infty}(A))$ such that
$p$ is a finite direct sum of projections of $\overline{AaA}$, $q$
is a finite direct sum of projections of $\overline{AbA}$ and there
is some $m \in \mathbb{N}$ such that:
\begin{equation}
(a - \varepsilon_{0})_{+} \precsim p\label{3.2}
\end{equation}
\begin{equation}
p \precsim a \otimes 1_{m}\label{3.3}
\end{equation}
\begin{equation}
(b - \varepsilon_{0})_{+} \precsim q\label{3.4}
\end{equation}
\indent Using \eqref{3.1}, \eqref{3.3} and \eqref{3.4}, we get:
\begin{eqnarray*}
\tau_{i}(p)&=&d_{{\tau_{i}}}(p) \leq d_{{\tau_{i}}}(a \otimes 1_{m}) =  md_{{\tau_{i}}}(a) < md_{{\tau_{i}}}((b - \varepsilon_{0})_{+}) \\
   \mbox{} &=& d_{{\tau_{i}}}((b - \varepsilon_{0})_{+} \otimes
1_{m}) \leq d_{{\tau_{i}}}(q \otimes 1_{m}) = \tau_{i}(q \otimes
1_{m}), 1 \leq i \leq l
\end{eqnarray*}
\noindent which obviously implies that:
\[
\tau_{i}(p) < \tau_{i}(q \otimes 1_{m}), 1 \leq i \leq l\,.
\]
\indent Multiplying each of these $l$ inequalities with appropriate
positive numbers and then summing up all the inequalities, we have
that:
\begin{equation}
\tau(p) < \tau(q \otimes 1_{m}), \text{ for all } \tau \in QT(A)\label{3.5}
\end{equation}
\noindent (of course, we used here the fact that $\tau_{1}, \tau_{2}, \ldots
, \tau_{l}$ are the extremal quasitraces of $A$).\\
\indent Since $A$ has strict comparison of projections, \eqref{3.5}
implies that:
\begin{equation}
p \precsim q \otimes 1_{m}\label{3.6}
\end{equation}
\indent Using again condition (ii) in Lemma \ref{lem:2.4} and the fact that $q$ is a
finite direct sum of projections of $\overline{AbA}$, it follows
that:
\begin{equation}
q \precsim b \otimes 1_{n}\label{3.7}
\end{equation}
\noindent for some $n \in \mathbb{N}$.\\
\indent Using \eqref{3.2}, \eqref{3.6} and \eqref{3.7} we obtain:
\[
(a - \varepsilon_{0})_{+} \precsim p \precsim q \otimes 1_{m}
\precsim b \otimes 1_{mn}
\]
\noindent which implies that:
\begin{equation}
(a - \varepsilon_{0})_{+} \precsim b \otimes 1_{mn}\,.\label{3.8}
\end{equation}
\indent Observe now that for all $0 < \varepsilon <
\varepsilon_{0}$, we have $(b - \varepsilon_{0})_{+} \leq (b -
\varepsilon)_{+}$ which implies that $(b - \varepsilon_{0})_{+}
\precsim (b - \varepsilon)_{+}$ and hence $d_{\tau}((b -
\varepsilon_{0})_{+}) \leq d_{\tau}((b - \varepsilon)_{+})$, for all
$\tau \in QT(A)$. Using this fact and \eqref{3.1}, we get:
\[
d_{\tau_{i}}(a) < d_{\tau_{i}}((b - \varepsilon)_{+}), \text{ for all } 0 <
\varepsilon < \varepsilon_{0}, 1 \leq i \leq l\nonumber
\]
\indent Working as above, we obtain that for each $0 < \varepsilon <
\varepsilon_{0}$ there exists $k = k(\varepsilon) \in \mathbb{N}$
such that:
\[
(a - \varepsilon)_{+} \precsim b \otimes 1_{k}\,.
\]
\indent Finally, for every $\varepsilon > \varepsilon_{0}$ we have
$(a - \varepsilon)_{+} \leq (a - \varepsilon_{0})_{+}$, which
implies that $(a - \varepsilon)_{+} \precsim (a -
\varepsilon_{0})_{+}$, and hence, using also \eqref{3.8}, we obtain:
\[
(a - \varepsilon)_{+} \precsim b \otimes 1_{mn}\,.
\]
\indent This ends the proof.
\end{proof}
\begin{remark}
\label{rem:comment} {\rm For all $a$ and $b\in A^+$, the conclusion
of the above result that, for every $\varepsilon>0$, there is $m$
such that $(a-\varepsilon)_+\precsim b\otimes 1_m$ can be rephrased
by saying that $a\in \overline{AbA}$, as mentioned in \cite[comment
before Lemma 4.1]{Ror:Z-absorbing}. We shall be using this below.}
\end{remark}

If $A$ is a unital $C^*$-algebra, we will denote by $T(A)$ the set of tracial states of $A$.
\begin{corollary}
\label{cor:3.4}
Let $A$ be a unital $AH$ algebra with the ideal property and
with finitely many extremal tracial states, and let $a, b \in M_{\infty}(A)^{+}$ such that
$d_{\tau}(a) < d_{\tau}(b)$, for all $\tau \in T(A)$. Then, for every $\varepsilon > 0$ there is $m \in \mathbb{N}$
such that $(a - \varepsilon)_{+} \precsim b \otimes 1_{m}$.
\end{corollary}
\begin{proof}
Let $U$ be an arbitrary $UHF$ algebra, that is, $U$ is the inductive limit of a
sequence $\{M_{n(k)}\}_{k}$ of matrix algebras via unital $\ast$-homomorphisms
$M_{n(k)} \rightarrow M_{n(k + 1)}$. Define $B := A \otimes U$. Then, clearly,
$B$ is a unital $AH$ algebra with the ideal property and with slow dimension growth
(in the sense of Gong (\cite{Gong:classII})). Then,
since $B$ is a unital exact $C^*$-algebra, by a theorem of Haagerup (\cite{Haa:quasi})
we have that $QT(B) = T(B)$. Hence,
by \cite[Theorem 5.1(b)]{Pas:ideal_prop_shape} it follows that $B$ has strict
comparison of projections. Let
$T(U) = \{\sigma\}$. Then, we have that $T(B) = T(A \otimes U) = \{\tau \otimes \sigma: \tau \in T(A)\}$. Therefore,
clearly, $B = A \otimes U$ has finitely many extremal tracial states, since $A$
has finitely many extremal
tracial states. We may assume, without any loss of generality, that $a, b \in A^{+}$. Define
$\widetilde{a} := a \otimes 1$, $\widetilde{b} := b \otimes 1 \in (A \otimes U)^{+} = B^{+}$. Then, for all
$\rho = \tau \otimes \sigma \in T(B) (\tau \in T(A)))$, we have by hypothesis:
\[
d_{\rho}(\widetilde{a}) = d_{\tau}(a) < d_{\tau}(b) = d_{\rho}(\widetilde{b})\,.
\]
\indent Using now Theorem \ref{thm:3.3} and Remark \ref{rem:comment} we deduce that
$a \otimes 1 = \widetilde{a} \in \overline{B\widetilde{b}B} = \overline{(A \otimes U)(b \otimes 1)(A \otimes U)}$,
from which we easily conclude (using, e.g., Fubini maps) that $a \in \overline{AbA}$, which implies the conclusion,
by \cite[Proposition 2.7(v)]{KirRor:pi}.
\end{proof}
The next two theorems are the main results of this section. Recall first that a positive ordered
abelian semigroup $W$ (in particular, the Cuntz semigroup of a $C^*$-algebra) is said to be
{\it almost unperforated} if for all $x, y \in W$ and all $m, n \in \mathbb{N}$, with $nx\leq my$ and $n > m$,
one has $x \leq y$ (see, e.g. \cite{Ror:Z-absorbing}).
\begin{theorem}
\label{thm:3.5} Let $A$ be a unital
$C^*$-algebra with the ideal property. Assume moreover that
$A$ has strict comparison of projections and  finitely many extremal quasitraces and that $W(A)$
is almost unperforated. Then $A$ has weak strict comparison.
\end{theorem}
\begin{proof} Use Theorem \ref{thm:3.3}, Remark \ref{rem:comment} and
\cite[Corollary 4.7]{Ror:Z-absorbing}.
\end{proof}
\begin{theorem}
\label{thm:3.6}
Let $A$  be a unital $AH$ algebra with the ideal property and with finitely many
extremal tracial states and such  that $W(A)$ is almost unperforated. Then $A$ has weak strict comparison.
\end{theorem}
\begin{proof}
Use Corollary \ref{cor:3.4}, Remark \ref{rem:comment} and
\cite[Corollary 4.7]{Ror:Z-absorbing}.
\end{proof}
\begin{theorem}
\label{thm:3.7}
 Let $A$ be a unital $AH$ algebra with the ideal property and with finitely many extremal tracial
states and let $B$ be a unital, simple, infinite dimensional $AH$ algebra with no dimension growth and with
a unique tracial state. Then $A \otimes B$  has weak strict comparison.
\end{theorem}
\begin{proof}
Observe first that since both $A$ and $B$ have the ideal property and $A$ (or $B$) is exact, it follows that
$A \otimes B$ has the ideal property (use, e.g., \cite[Corollary
1.3]{PasRor:tensIP}). On the other hand, by a result in
\cite{TomWin:ASH}, $B$ is $\mathcal{Z}$-stable, that is $B \cong B
\otimes \mathcal{Z}$, where $\mathcal{Z}$ is the Jiang-Su algebra
(\cite{JiaSu:Z}). Hence the unital $AH$ algebra with the ideal
property $A \otimes B$ is $\mathcal{Z}$-stable, i.e. $A \otimes B
\cong (A \otimes B) \otimes \mathcal{Z}$, and then \cite[Theorem
4.5]{Ror:Z-absorbing} implies that $W(A \otimes B)$ is almost
unperforated. Note that if $T(B) = \{\sigma\}$, then $T(A \otimes B)
= \{\tau \otimes \sigma: \tau \in T(A)\}$ and since $A$ has finitely
many extremal tracial states, it is obvious that $A \otimes B$ has
also finitely many extremal tracial states. Now,
the fact that $A \otimes B$ has weak strict comparison follows
from Theorem \ref{thm:3.6}.
\end{proof}

\begin{remark}
\label{rem:revision} {\rm  We may say that a unital C$^*$-algebra $A$ with
$QT(A)\neq\emptyset$ has \emph{almost weak strict comparison} if $A$
satisfies all the conditions in the definition of weak strict
comparison (see Definition \ref{dfn:3.1}) with the only difference
that the condition:
\[
(*)\,\, d(a) < d(b) \text{ for all } d \in E\
\]
is replaced by the new condition:
\[
(**)\,\,\text{there is } \varepsilon_{0} > 0 \text{ such that } d(a)
< d((b- \varepsilon_{0})_{+}) \text{ for all } d \in E\,,
\]
with $E$ as in Definition \ref{dfn:3.1} above (of course, we still request that $d(a) < d(b)$ for every
$d \in \{f \in DF(A) \smallsetminus E:f(b) = 1\}$).

In the proof of Theorem \ref{thm:3.3} we showed, in particular, that
in the case when a unital C$^*$-algebra $A$ has finitely many
extremal quasitraces, then $(*)\implies (**)$. Therefore, in this case, if $A$ has
almost weak strict comparison, it follows that $A$ has weak strict comparison.
Note that if we drop the condition that the C$^*$-algebra $A$ has
finitely many extremal quasitraces (tracial states), the conclusions
of Theorem \ref{thm:3.3} and of Corollary \ref{cor:3.4} remain true
if we replace in their hypotheses condition $(*)$ by condition
$(**)$ as above. Also, it is easy to see that, if in Theorems
\ref{thm:3.5}, \ref{thm:3.6}, \ref{thm:3.7} we drop the condition
that $A$ has finitely many extremal quasitraces (tracial states) and
the condition that B has a unique tracial state (in Theorem
\ref{thm:3.7}), then they remain true if we replace in their
conclusions ``weak strict comparison'' by ``almost weak strict
comparison'' (to show these results we use the same proofs). In
conclusion, we thus obtain generalizations of all the results proved
in Section 3}.
\end{remark}

\section*{Acknowledgements}

The second named author was partially supported by a MEC-DGESIC
grant (Spain) through Project MTM2008-06201-C02-01/MTM, and by the
Comissionat per Universitats i Recerca de la Generalitat de
Catalunya. Part of this work was carried out during the AIM Workshop
``The Cuntz semigroup'', held in November 2009. The authors are
grateful to AIM and its staff for their support and working
conditions provided. It is also a pleasure to thank Aaron Tikuisis
for comments on a first draft that allowed us to simplify some
arguments and led to an improved exposition.

\bibliographystyle{amsplain}
%\bibliography{operator}
\providecommand{\bysame}{\leavevmode\hbox to3em{\hrulefill}\thinspace}
\providecommand{\MR}{\relax\ifhmode\unskip\space\fi MR }
\providecommand{\MRhref}[2]{%
  \href{http://www.ams.org/mathscinet-getitem?mr=#1}{#2}
}
\providecommand{\href}[2]{#2}

\end{document}